\UseRawInputEncoding
\documentclass[a4paper,reqno,10pt]{amsart}

\usepackage{amsrefs} 
\usepackage{mathrsfs}                   
\usepackage{stmaryrd}                   
\usepackage{enumerate}
\usepackage{mathtools}
\usepackage[bookmarks=true]{hyperref}   
\usepackage{todonotes}
\usepackage{nicefrac}
\usepackage{caption}
\usepackage{array,ragged2e}

\hypersetup{
    colorlinks,
    linkcolor={red!50!black},
    citecolor={blue!50!black},
    urlcolor={blue!80!black}
}

\makeatletter
\renewcommand*\env@matrix[1][*\c@MaxMatrixCols c]{%
   \hskip -\arraycolsep
   \let\@ifnextchar\new@ifnextchar
   \array{#1}}
\makeatother

\newenvironment{smallpmatrix}
  {\left(\begin{smallmatrix}}
  {\end{smallmatrix}\right)}
  
\newtagform{simple}{}{}{}
\author{Christian B\"ar}
\author{Rafe Mazzeo}
\title{Manifolds with many Rarita-Schwinger fields}
\date{\today}

\address{Christian B\"ar, 
Institut für Mathematik,
Universit\"at Potsdam, 
D-14476, Potsdam, Germany
}
\urladdr{\href{https://www.math.uni-potsdam.de/baer}{https://www.math.uni-potsdam.de/baer}}
\email{\href{mailto:cbaer@uni-potsdam.de}{cbaer@uni-potsdam.de}}

\address{Rafe Mazzeo, 
Department of Mathematics,
450 Serra Mall,
Stanford, CA 94305-2125, USA}
\urladdr{\href{http://web.stanford.edu/~rmazzeo/cgi-bin/.}{http://web.stanford.edu/~rmazzeo/cgi-bin/.}}
\email{\href{mailto:rmazzeo@stanfod.edu}{rmazzeo@stanford.edu}}


\newcommand{\<}{\langle}
\renewcommand{\>}{\rangle}
\newcommand{\R}{\mathbb{R}}
\newcommand{\Q}{\mathbb{Q}}
\newcommand{\C}{\mathbb{C}}
\newcommand{\N}{\mathbb{N}}
\newcommand{\Z}{\mathbb{Z}}
\newcommand{\DD}{\mathscr{D}}
\newcommand{\HH}{\mathscr{H}}
\newcommand{\NN}{\mathscr{N}}
\newcommand{\Sh}{\widehat\Sigma}
\newcommand{\Spin}{\mathsf{Spin}}
\newcommand{\eps}{\varepsilon}
\newcommand{\myicon}{$\,\,\,\triangleright$}
\DeclareMathOperator{\id}{id}
\DeclareMathOperator{\Ric}{Ric}
\DeclareMathOperator{\ind}{index}
\DeclareMathOperator{\Ahat}{\hat A}
\DeclareMathOperator{\ch}{ch}
\DeclareMathOperator{\RS}{RS}
\DeclareMathOperator{\PS}{PS}
\DeclareMathOperator{\coker}{coker}
\DeclareMathOperator{\OO}{O}
\DeclareMathOperator{\coeff}{coeff}

 \newtheorem{thm}{Theorem}[section]
 \newtheorem{lemma}[thm]{Lemma}
 \newtheorem{cor}[thm]{Corollary}
 \newtheorem{prop}[thm]{Proposition}

 \theoremstyle{definition}
 \newtheorem{definition}[thm]{Definition}
 
 \newtheorem{remark}[thm]{Remark}

\begin{document}

\begin{abstract}
The Rarita-Schwinger operator is the twisted Dirac operator restricted to $\nicefrac32$-spinors.  
Rarita-Schwinger fields are solutions of this operator which are in addition divergence-free.  
This is an overdetermined problem and solutions are rare; it is even more unexpected for there to be large dimensional spaces of solutions. 

In this paper we prove the existence of a sequence of compact manifolds in any given dimension greater than or equal to $4$ for which the dimension of the space of Rarita-Schwinger fields tends to infinity.  
These manifolds are either simply connected K\"ahler-Einstein spin with negative Einstein constant, or products of such spaces with flat tori. 
Moreover, we construct Calabi-Yau manifolds of even complex dimension with more linearly independent Rarita-Schwinger fields than flat tori of the same dimension.
\end{abstract} 
\maketitle

\section{Introduction}
The goal of this note is to establish the existence of a sequence of compact Riemannian spin manifolds in any fixed dimension $n \geq 4$ for which the number of Rarita-Schwinger fields tends to infinity.  
When $n$ is divisible by $4$, these manifolds are K\"ahler-Einstein with negative Einstein constant, and arise as complete intersections in some higher dimensional complex projective
space. For other values of $n$ we take products of such K\"ahler-Einstein spaces with flat tori. The Rarita-Schwinger operator 
$Q$ is the twisted Dirac operator acting on $\nicefrac32$-spinors, and perhaps somewhat
confusingly, Rarita-Schwinger fields are solutions of this twisted Dirac operator which are also divergence free.
This is an overdetermined right-elliptic operator and so the existence of nontrivial solutions is nongeneric.  

This operator $Q$ appeared initially in the 1941 paper of Rarita and Schwinger \cite{RS1941} to describe the wave functions of particles of spin $\nicefrac{3}{2}$, and has been used extensively in the physics literature since then. In mathematics it is relatively unstudied, and has appeared mostly
as one example amongst many in the family of all twisted Dirac operators. 
Notably, ￼Semmelmann computed its spectrum on spheres and complex projective spaces \cite{Sdiss}.
Bure\v{s} \cite{B1999} introduced it
from the point of representation theory, studied some basic examples of solutions and computed its index
in general. Branson and Hijazi \cite{BH2002}  proved its conformal covariance and determined the Weitzenb\"ock 
formula for $Q^2$. Unlike the simpler Dirac
operator, this Weitzenb\"ock formula has lower order terms which do not have a sign under standard geometric
hypotheses, so the existence of solutions (divergence-free or not) does not have any immediate connection to
geometry.  M.\ Wang \cite{W1991} studied the role of Rarita-Schwinger fields in the deformation theory of Einstein 
metrics with parallel spinors, see also his earlier paper \cite{W1989}. 
The problem of counting Rarita-Schwinger fields was considered only quite recently by Homma and Semmelmann 
\cite{HS2018}, with related work more recently still in Homma--Tomihisa \cite{HT2020}. The goal of \cite{HS2018}
was to find manifolds admitting {\it any} nontrivial Rarita-Schwinger fields; as part of this they obtain a complete 
classification of positive quaternion-K\"ahler manifolds and spin symmetric 
spaces admitting such fields, but their mention of the negative Einstein case is rather brief.  We refer to all of
these papers for a more extended discussion and some description of the use of Rarita-Schwinger fields in
physics. 

Our main observation is that the negative K\"ahler-Einstein case provides a particularly rich set of examples. 
To set the stage, suppose that $(M,g)$ is a Riemannian manifold carrying a spin structure. Denote by $\Sigma M$
its spin bundle, with its chirality decomposition $\Sigma^+ M \oplus \Sigma^-M$ when $\dim M$ is even.  
The standard Dirac operator acts on sections of $\Sigma M$; coupling to the Levi-Civita connection on $TM$ we may 
then define the twisted
Dirac operator $\DD$ acting on sections of $\Sigma M \otimes T M$.  As we explain in the next section,
this bundle splits as the direct sum of an isometric copy of $\Sigma M$ and another bundle $\widehat{\Sigma} M$,
which is the bundle of $\nicefrac32$-spinors.   The Rarita-Schwinger operator $Q$ is the restriction of $\DD$ to 
$\widehat{\Sigma} M$, i.e., the projection of $\DD \psi$ to $\widehat{\Sigma} M$, where $\psi \in 
\Gamma(\widehat{\Sigma} M)$.  The complementary projection of $\DD \psi$ onto the copy of $\Sigma M$
in $\Sigma M \otimes TM$ is essentially the divergence of $\psi$. The space of Rarita-Schwinger fields is the set 
of such sections $\psi$ for which both components of $\DD \psi$ vanish, and we define
\[
\RS(M) = \dim \{\psi \in \Gamma( \widehat{\Sigma}M): \DD \psi = 0\}.
\]
For even dimensional manifolds we can also define $\RS^\pm(M)$, the dimension of the space of even or odd chirality
Rarita-Schwinger fields.  Our main result is 

\begin{thm} 
Let $n$ be any positive integer greater than or equal to $4$ and $C$ any positive constant. 
Then there exists a compact Riemannian spin manifold $M^n$ such that $\RS(M) > C$. 
If $n$ is divisible by $4$, then we can take $M$ to be a simply connected compact K\"ahler-Einstein manifold of complex dimension $m = n/2$ which is spin and has negative Einstein constant. 
\end{thm}

This is a corollary of a more precise theorem which, in the case where $M$ is K\"ahler-Einstein and spin with negative
Einstein constant, estimates from below $\RS(M)$ (and $\RS^\pm(M)$) as the difference of a certain characteristic number 
of $M$ and a number which is essentially the dimension of the space of parallel spinors, see Theorem~\ref{thm:LowerBoundRS}.

Moreover, for each even $m\ge2$ we find a complex $m$-dimensional simply connected compact Calabi-Yau manifold with many Rarita-Schwinger fields.
This means that the dimension of the space of these fields is much larger than that of flat tori of the same dimension, see Corollary~\ref{cor:CYlowerbound} for a precise formulation.

The key tool in the proof is a certain elliptic complex involving the Rarita-Schwinger operator which is only well-defined on Einstein manifolds. 
The index of this complex leads to the aforementioned estimate of $\RS(M)$.  

The existence of arbitrarily many linearly independent Rarita-Schwinger fields in a given dimension provides a stark contrast with the 
behavior of other better-known overdetermined elliptic operators which appear in geometry. For example, the Killing operator is said to be
overdetermined of finite type, which corresponds to the fact that there is a sharp bound for the dimension of the space of solutions of 
the Killing operator which depends only on the dimension of the manifold. The dimension of the space of twistor spinors can
be bounded similarly. Our results show that the Rarita-Schwinger operator is not of this type. This finite type phenomenon was 
studied exhaustively in the work of Kodaira and Spencer. 

\emph{Acknowledgments.}
This work was primarily carried out during a visit of the first-named author to Stanford University in May 2019.
He wishes to thank the Stanford Math Research Center for partial financial support and hospitality during this visit.
He was also supported by SPP 2026 funded by Deutsche Forschungsgemeinschaft. 
The second author was supported by the NSF grant DMS-1608223.

\section{Rarita-Schwinger fields}

We start by describing the geometric setup for the study of Rarita-Schwinger fields.
For an introduction to general aspects of spin geometry see e.g.\ \cite{LM}.

\subsection{The algebra of vector-spinors}

We equip $\R^n$ with its standard Euclidean metric $\<\cdot,\cdot\>$ and standard orientation.
Let $\Sigma_n$ be the complex spinor module of the spin group $\Spin(n)$.
Then there is a linear map $\gamma:\R^n\otimes\Sigma_n \to \Sigma_n$ satisfying the \emph{Clifford relations} 
\begin{equation*}
Y\cdot (X\cdot\phi) + X\cdot (Y\cdot\phi) + 2\<X,Y\>\phi = 0
\end{equation*}
where we have used the notation $X\cdot\phi = \gamma(X\otimes\phi)$.
We denote the kernel of $\gamma$ by $\Sh_n \subset \R^n\otimes\Sigma_n$.
The Euclidean metric on $\R^n$ and the Hermitean metric on $\Sigma_n$ induce a Hermitean metric on $\R^n\otimes\Sigma_n$ and hence on $\Sh_n$.

Let $e_1,\ldots,e_n$ be the standard basis of $\R^n$.
We define
\begin{equation*}
\iota: \Sigma_n \to \R^n\otimes\Sigma_n,
\quad
\iota(\phi) = -\frac1n \sum_{j=1}^n e_j \otimes e_j\cdot \phi.
\end{equation*}
The factor $-\tfrac1n$ is chosen such that $\gamma\circ\iota = \id$, hence $\iota\circ\gamma$ is a projection.
The complementary projection $\hat\pi = \id - \iota\circ\gamma$ has image $\Sh_n$.
The maps $\gamma$, $\iota$, and $\hat\pi$ are $\Spin(n)$-equivariant.
One easily checks that 
\begin{equation*}
|\iota(\phi)|^2 = \frac1n |\phi|^2,
\end{equation*}
in particular, $\iota$ is injective, and that
\begin{equation*}
\gamma^* = n\iota.
\end{equation*}
In particular, $\iota\circ\gamma$ and hence $\hat\pi$ are self-adjoint projections.
We obtain the orthogonal decomposition of $\Spin(n)$-modules $\R^n\otimes\Sigma_n = \iota(\Sigma_n)\oplus\Sh_n$.

The metric on $\R^n$ yields a contraction $\R^n\otimes(\iota(\Sigma_n)\oplus\Sh_n) = \R^n\otimes\R^n\otimes\Sigma_n \to \Sigma_n$ which we denote by $X\otimes \Phi \mapsto \Phi(X)$, i.e.\ $(\sum_{j=1}^n e_j\otimes \phi_j)(X) = \sum_{j=1}^n \<e_j,X\>\phi_j$.

If $n$ is even, $\Sigma_n$ decomposes as a $\Spin(n)$-module into spinors of positive and negative chirality, $\Sigma_n = \Sigma_n^+ \oplus \Sigma_n^-$.
Clifford multiplication interchanges chirality, $\gamma(\R^n\otimes\Sigma_n^\pm)=\Sigma_n^\mp$.
We get a corresponding splitting
\begin{equation*}
\R^n\otimes\Sigma_n^\pm = \iota(\Sigma_n^\mp) \oplus \Sh_n^\pm . 
\end{equation*}
Now let $M$ be an $n$-dimensional Riemannian spin manifold.
Associating the $\Spin(n)$-modules $\R^n$, $\Sigma_n$, and $\Sh_n$ to the spin structure of $M$, we obtain the tangent bundle $TM$, the spinor bundle $\Sigma M$ and the bundle of $\nicefrac32$-spinors $\Sh M$, respectively.
If $n$ is even, we furthermore have the bundles $\Sigma^\pm M$ and $\Sh^\pm M$ of spinors and $\nicefrac32$-spinors of positive and negative chirality.

Since $\gamma$, $\iota$, and $\hat\pi$ are $\Spin(n)$-equivariant, they induce vector bundle morphisms which we again denote by $\gamma$, $\iota$, and $\hat\pi$.

\subsection{Dirac, twistor, and Rarita-Schwinger operator}

The Levi-Civita connection on $TM$ induces one on $\Sigma M$ and hence on $\Sh M$.
We will denote all these connections by $\nabla$.
Note that the connection on the spinor bundle is a differential operator $\nabla:\Gamma(\Sigma M) \to \Gamma(TM\otimes\Sigma M)$.
The \emph{Dirac operator} is defined as 
$$
D=\gamma\circ\nabla:\Gamma(\Sigma M)\to\Gamma(\Sigma M)
$$ 
and the \emph{twistor} or \emph{Penrose operator} as 
$$
P=\hat\pi\circ\nabla:\Gamma(\Sigma M)\to\Gamma(\Sh M).
$$
The Dirac operator is formally self-adjoint, $D=D^*$, while the adjoint of $P$ is given on $\nicefrac32$-spinors $\Phi$ by
$$
P^*\Phi = (\nabla^*\circ\hat\pi)(\Phi) = \nabla^*\Phi = -\sum_{j=1}^n (\nabla_{e_j}\Phi)(e_j).
$$
Now let $\DD:\Gamma(TM\otimes\Sigma M)\to\Gamma(TM\otimes\Sigma M)$ be the twisted Dirac operator acting on vector-spinor fields.
We write it as a $2\times2$-matrix with respect to the decomposition $TM\otimes\Sigma M = \iota(\Sigma M)\oplus\Sh M$:
\renewcommand{\arraystretch}{1.5}
\begin{equation*}
\DD =
\begin{pmatrix}[c|c]
\frac{2-n}{n}\iota\circ D\circ\iota^{-1} & 2\iota\circ P^*\\ \hline 
\frac2n P\circ\iota^{-1}                 & Q
\end{pmatrix}.
\end{equation*}
\renewcommand{\arraystretch}{1}%
This follows from straight-forward computation, see also \cite{W1991}*{Prop.~2.7(b)}.
The part $Q:\Gamma(\Sh M)\to \Gamma(\Sh M)$ is called the \emph{Rarita-Schwinger operator}.
It is a formally self-adjoint first order elliptic differential operator.

If $n$ is even, the Dirac operators $D$ and $\DD$ interchange chirality, hence we get operators $D^\pm:\Gamma(\Sigma^\pm M)\to\Gamma(\Sigma^\mp M)$ and $\DD^\pm:\Gamma(TM \otimes \Sigma^\pm M)\to\Gamma(TM \otimes \Sigma^\mp M)$.
The splitting $TM\otimes\Sigma^\pm M = \iota(\Sigma^\mp M)\oplus\Sh^\pm M$ leads to 
\renewcommand{\arraystretch}{1.5}
\begin{equation*}
\DD^\pm =
\begin{pmatrix}[c|c]
\frac{2-n}{n}\iota\circ D^\mp\circ\iota^{-1} & 2\iota\circ (P^\pm)^*\\ \hline 
\frac2n P^\mp\circ\iota^{-1}                 & Q^\pm
\end{pmatrix}.
\end{equation*}
\renewcommand{\arraystretch}{1}%
(The reader should beware here: by our conventions, the formula for $\DD^\pm$ contains $D^\mp$ in its upper left entry,
which is perhaps awkward, but is coherent with other notation.) 
Thus the Rarita-Schwinger operator also interchanges chirality, $Q^\pm:\Gamma(\Sh^\pm M)\to\Gamma(\Sh^\mp M)$, while the twistor operator $P^\pm:\Gamma(\Sigma^\pm M)\to\Gamma(\Sh^\pm M)$ and its adjoint $(P^\pm)^*:\Gamma(\Sh^\pm M)\to\Gamma(\Sigma^\pm M)$ preserve it.

\subsection{An exact sequence on Einstein manifolds}

As it turns out, the Rarita-Schwinger operator has some especially nice features on Einstein manifolds.

\begin{prop}
Let $M$ be an Einstein spin manifold of dimension $n\ge3$.
Then for every $\alpha\in\R$
\begin{equation}
0 \rightarrow \Gamma(\Sigma M)
\xrightarrow{\begin{smallpmatrix} \iota D  \\ \alpha P \end{smallpmatrix}}
\Gamma(TM\otimes\Sigma M)
\xrightarrow{(\alpha\frac{2-n}{n}P\iota^{-1},\, Q) }
\Gamma(\Sh M) \rightarrow 0
\label{eq:EllComplex}
\end{equation}
is an elliptic complex.

If $n$ is even, then \eqref{eq:EllComplex} restricts to elliptic complexes
\begin{equation}
0 \rightarrow \Gamma(\Sigma^\pm M)
\xrightarrow{\begin{smallpmatrix}\iota D^\pm \\ \alpha P^\pm \end{smallpmatrix}}
\Gamma(TM\otimes\Sigma M^\pm)
\xrightarrow{( \alpha\frac{2-n}{n}P^\mp\iota^{-1},\, Q^\pm)}
\Gamma(\Sh^\mp M) \rightarrow 0
\label{eq:EllComplexChiral}
\end{equation}
If, in addition, $M$ is compact without boundary, then the Euler number of \eqref{eq:EllComplexChiral} is given by 
\begin{equation}
\chi^\pm(M) = \mp \<\Ahat(M) \ch(T^\C M),[M]\> .
\end{equation}
\end{prop}

Here $\Ahat(M)$ is the $\Ahat$-class of $M$, $\ch(T^\C M)$ is the Chern character of the complexified tangent bundle, and $[M]$ is the fundamental cycle in homology.

\begin{remark}
It seems reasonable to expect that \eqref{eq:EllComplex} and \eqref{eq:EllComplexChiral} might be the initial parts of a BGG
complex that exists in certain situations. To our knowledge, this does not seem to be the case.
\end{remark}

\begin{proof}
Using the Weitzenb\"ock formula for $\DD^2$ one computes
\begin{equation}
\tfrac{2-n}{n} PD + QP = \tfrac12\Ric^0,
\label{eq:EinsteinId}
\end{equation}
where $\Ric^0$ is the traceless part of Ricci curvature, see \cite{W1991}*{Prop.~2.9(b)}.
Since $M$ being Einstein means $\Ric^0=0$, we see that \eqref{eq:EllComplex} is a complex.
Then, for any covector $\xi\in T_x^*M$, the corresponding sequence of principal symbols 
\begin{equation*}
0 \rightarrow \Sigma_x M
\xrightarrow{\sigma\left(\begin{smallpmatrix}\iota D\\ \alpha P\end{smallpmatrix};\xi\right)}
T_xM\otimes\Sigma_x M
\xrightarrow{\sigma\left(\alpha\frac{2-n}{n}P\iota^{-1}, Q;\xi\right)}
\Sh_x M \rightarrow 0
\label{eq:SymbolSequence}
\end{equation*}
is a complex too.
Let $\xi\neq0$.
To show ellipticity of \eqref{eq:EllComplex} it suffices, for dimensional reasons, to show that the two principal symbols have maximal rank.
But this is clear: since $D$ is elliptic the first principal symbol is injective and since $Q$ is elliptic the second is surjective.

It remains to compute the Euler numbers of \eqref{eq:EllComplexChiral}.
Since Euler numbers are invariant under continuous deformations of elliptic complexes, it suffices to do this for $\alpha=0$.
Then \eqref{eq:EllComplexChiral} becomes
\begin{equation*}
0 \rightarrow \Gamma(\Sigma^\pm M)
\xrightarrow{\begin{smallpmatrix}D^\pm\\ 0\end{smallpmatrix}}
\Gamma(\Sigma M^\mp \oplus \Sh^\pm M)
\xrightarrow{\left(0, Q^\pm\right)}
\Gamma(\Sh^\mp M) \rightarrow 0
\end{equation*}
so that the Euler number is given by  
\begin{equation}
\chi^\pm(M) = \ind(D^\pm)-\ind(Q^\pm).
\label{eq:Euler1}
\end{equation}
For $\alpha\in\R$ consider the operator
\renewcommand{\arraystretch}{1.5}
\begin{equation*}
\DD_\alpha =
\begin{pmatrix}[c|c]
\frac{2-n}{n}\iota\circ D\circ\iota^{-1} & 2\alpha\iota\circ P^*\\ \hline 
\frac{2\alpha}n P\circ\iota^{-1}                 & Q
\end{pmatrix}.
\end{equation*}
\renewcommand{\arraystretch}{1}%
Again using \eqref{eq:EinsteinId}, we see that the off-diagonal terms of $\DD_\alpha^2$ vanish for Einstein manifolds and we get that 
\renewcommand{\arraystretch}{1.5}
\begin{equation*}
\DD_\alpha^2 =
\begin{pmatrix}[c|c]
\left(\frac{2-n}{n}\right)^2\iota\circ D^2\circ\iota^{-1} + \frac{4\alpha^2}{n}\iota P^*P\iota^{-1} & 0\\ \hline 
0                & \frac{4\alpha^2}n PP^* + Q^2
\end{pmatrix}
\end{equation*}
\renewcommand{\arraystretch}{1}%
is elliptic for every choice of $\alpha$.
Thus all $\DD_\alpha$ are elliptic.
Therefore
\begin{align*}
\ind(\DD^\pm)
&=
\ind(\DD_1^\pm)
=
\ind(\DD_0^\pm) \\
&=
\ind(D^\mp) + \ind(Q^\pm) \\
&=
\ind(Q^\pm) - \ind(D^\pm).
\end{align*}
Inserting this into \eqref{eq:Euler1} we get, using the Atiyah-Singer index theorem for twisted Dirac operators,
\begin{equation*}
\chi^\pm(M) = -\ind(\DD^\pm) = \mp \<\Ahat(M) \ch(T^\C M),[M]\> .
\qedhere
\end{equation*}
\end{proof}

\begin{remark}
The proof above shows that the index of the Rarita-Schwinger operator $Q^\pm$ is given by 
$$
\ind(Q^\pm) = \ind(\DD^\pm) + \ind(D^\pm) = \pm \<\Ahat(M) (\ch(T^\C M)+1),[M]\> .
$$
See \cite[Sec.~3]{HS2018} for a different derivation of this fact.
\end{remark}
\begin{remark}
We have presented this computation phrased using the elliptic complexes \eqref{eq:EllComplex} and \eqref{eq:EllComplexChiral}. 
It would be straightforward to wrap each of these complexes into a single elliptic operator (depending on the parameter $\alpha$) and compute its index.  
Neither approach seems shorter than the other and the end results are the same.
The complex has the advantage that its index can be refined by its cohomology which might turn out to be useful in some other context.
\end{remark}

\begin{definition}
A section $\psi\in\Gamma(\Sh M)$ is called a \emph{Rarita-Schwinger field} if $\DD\psi=0$.
\end{definition}
Note that this is equivalent to $Q\psi=0$ and $P^*\psi=0$.
Since $Q$ is elliptic, the Rarita-Schwinger equation is overdetermined.
By elliptic regularity theory, (distributional) Rarita-Schwinger fields are smooth.

We use the following notation for the dimension of the space of Rarita-Schwinger fields:
$$
\RS(M) := \dim\{\psi\in\Gamma(\Sh M) \mid \DD\psi=0\}
$$
and similarly, in even dimensions,
$$
\RS^\pm(M) := \dim\{\psi\in\Gamma(\Sh^\pm M) \mid \DD^\pm\psi=0\}
$$
for Rarita-Schwinger fields of positive and negative chirality, respectively.
We then have $\RS(M)=\RS^+(M)+\RS^-(M)$.

\begin{definition}\label{def:Nn}
For $n\in\N$ put
$$
\NN(n) :=
\begin{cases}
2^m & \text{ if }n=4m\text{ or }n=4m+7, \\
2^{m+1} & \text{ if }n=4m+ 2\cdot 7\text{ or }n=4m+3\cdot 7, \\
0    & \text{ else.}
\end{cases}
$$
\end{definition}

For low values of $n$ the number $\NN(n)$ is given by
\renewcommand{\arraystretch}{1.2}
\begin{center}
\begin{tabular}{|c||c|c|c|c|c|c|c|c|c|c|c|c|c|c|}
\hline
$n$ & 1 & 2 & 3 & 4 & 5 & 6 & 7 & 8 & 9 & 10 & 11 & 12 & 13 & 14\\ \hline
$\NN(n)$ & 0 & 0 & 0 & 2 & 0 & 0 & 1 & 4 & 0 & 0 & 2 & 8 & 0 & 2 \\ \hline\hline
$n$ & 15 & 16 & 17 & 18 & 19 & 20 & 21 & 22 & 23 & 24 & 25 & 26 & 27 & 28\\ \hline
$\NN(n)$ & 4 & 16 & 0 & 4 & 8 & 32 & 2 & 8 & 16 & 64 & 4 & 16 & 32 & 128 \\ \hline
\end{tabular}
\medskip
\captionof{table}{Maximal number of linearly independent parallel spinors on manifolds without flat factor}
\end{center}

\begin{lemma}\label{lem:maxdim}
Let $M$ be a complete simply connected Riemannian spin manifold of dimension $n$.
If the dimension of the space of parallel spinors is larger than $\NN(n)$ then $M$ has a flat factor, i.e.\ $M$ is isometric to a Riemannian product $N\times\R^k$ with $k\ge1$.
\end{lemma}

\begin{proof}
If $M$ is irreducible, i.e.\ does not split as an isometric product, and carries parallel spinors then $M$ must have one of the holonomy groups from the following table (see \cite{W1989}):
\begin{center}
\begin{tabular}{|c|c|c|}
\hline
$n$ & $\mathsf{Hol}(M)$ & $\dim\{\text{parallel spinors}\}$ \\ \hline\hline
$2m$ & $\mathsf{SU}(m)$ & $2$ \\ \hline
$4m$ & $\mathsf{Sp}(m)$ & $m+1$ \\ \hline
$7$ & $\mathsf{G}_2$ & $1$ \\ \hline
$8$ & $\mathsf{Spin}(7)$ & $1$ \\ \hline
\end{tabular}
\medskip
\captionof{table}{Holonomy groups of irreducible manifolds with parallel spinors}
\end{center}
If $M$ is reducible and admits parallel spinors but does not contain a flat factor then $M$ must be isometric to a product of manifolds of the type listed in the table. 
For a Riemannian product $M=M_1\times M_2$ the spinor bundle of $M$ can be naturally identified with the tensor product of the spinor bundles of $M_1$ and $M_2$ unless both $M_1$ and $M_2$ are odd-dimensional.
In that case the spinor bundle of $M$ is the sum of two copies of the tensor product.
The connection on the spinor bundle of $M$ coincides with the tensor product connection.
Hence the space of parallel spinors on $M$ is isomorphic to the tensor product of those on $M_1$ and on $M_2$ or the sum of two copies of it if both $M_1$ and $M_2$ are odd-dimensional. 

Let $\tilde\NN(n)$ be the maximal dimension of the space of parallel spinors on an $n$-dimensional product of manifolds as in the table.
Factors with holonomy $\mathsf{SU}(m)$ ($m\ge3$), $\mathsf{Sp}(m)$ ($m\ge 2$) or $\mathsf{Spin}(7)$ will not occur in such a maximal product because products of K3-surfaces (holonomy $\mathsf{SU}(2)$) yield a higher-dimensional space of parallel spinors.
More than three $\mathsf{G}_2$-factors will not occur either because four such factors can be replaced by seven $\mathsf{SU}(2)$-factors, again yielding a higher-dimensional space of parallel spinors.

Thus $\tilde\NN(n)$ is realized by products of K3-surfaces and up to three $\mathsf{G}_2$-manifolds.
This shows that $\tilde\NN(n)=\NN(n)$ as given in Definition~\ref{def:Nn}.
The additional factor of $2$ in the case of two or three $\mathsf{G}_2$-factors is due to the space of parallel spinors on the product of two (odd-dimensional!) $\mathsf{G}_2$-manifolds being $2$-dimensional.
\end{proof}

\begin{remark}
The proof shows that the bound in Lemma~\ref{lem:maxdim} is sharp.
There exist complete simply connected Riemannian spin manifolds $X$ and $Y$ with holonomy $\mathsf{SU}(2)$ and $\mathsf{G}_2$, respectively.
Then 
$$
M = \underbrace{X \times \ldots \times X}_{m} \times \underbrace{Y \times \ldots \times Y}_{0,1,2,\text{ or }3}
$$
does not have a flat factor and the space of parallel spinors on $M$ has precisely dimension $\NN(\dim(M))$.
\end{remark}

\begin{thm}\label{thm:LowerBoundRS}
Let $M$ be compact Einstein spin manifold without boundary and of even dimension $n\ge4$.
Then
$$
\RS^\pm(M) \ge
\begin{cases}
\pm \<\Ahat(M) \ch(T^\C M),[M]\> - \NN(n)& \text{if $M$ is Ricci-flat,} \\
\pm \<\Ahat(M) \ch(T^\C M),[M]\> & \text{otherwise,} 
\end{cases}
$$
and 
$$
\RS(M) \ge
\begin{cases}|\<\Ahat(M) \ch(T^\C M),[M]\>| - \NN(n)& \text{if $M$ is Ricci-flat,} \\
|\<\Ahat(M) \ch(T^\C M),[M]\>| & \text{otherwise.} 
\end{cases}
$$
\end{thm}

\begin{proof}
The elliptic complex \eqref{eq:EllComplexChiral} with $\alpha=1$ reads as
\begin{equation*}
0 \rightarrow \Gamma(\Sigma^\pm M)
\xrightarrow{\nabla}
\Gamma(TM\otimes\Sigma M^\pm)
\xrightarrow{\begin{smallpmatrix}\frac{2-n}{n}P^\mp\iota^{-1}\\ Q^\pm\end{smallpmatrix}}
\Gamma(\Sh^\mp M) \rightarrow 0 .
\end{equation*}
By Hodge theory, 
\begin{align*}
\chi^\pm(M)
&\le
\dim\ker(\nabla) + \dim\coker\begin{smallpmatrix}\frac{2-n}{n}P^\mp\iota^{-1}\\ Q^\pm\end{smallpmatrix} \\
&=
\dim\ker(\nabla) + \dim\ker\begin{smallpmatrix}\frac{2-n}{n}P^\mp\iota^{-1}\\ Q^\pm\end{smallpmatrix}^* \\
&=
\dim\ker(\nabla) + \dim\ker((2-n)\iota (P^\mp)^*, Q^\mp) \\
&=
\dim\ker(\nabla) + \RS^\mp(M).
\end{align*}
Thus
$$
\RS^\pm(M) \ge \chi^\mp(M) - \dim\ker(\nabla) = \pm \<\Ahat(M) \ch(T^\C M),[M]\> - \dim\ker(\nabla).
$$
If $M$ has nontrivial parallel spinors, $M$ must be Ricci flat. 
In this case, either $\dim\ker(\nabla) \le \NN(n)$ which yields the claim or else $\dim\ker(\nabla) > \NN(n)$ and then the universal covering of $M$ must contain a flat factor by Lemma~\ref{lem:maxdim}.
Then all characteristic numbers of $M$ (including $\<\Ahat(M) \ch(T^\C M),[M]\>$) vanish and the claim holds trivially.

If $M$ is not Ricci-flat then $\dim\ker(\nabla) = 0$.
This concludes the proof.
\end{proof}

\section{Complete intersections}

Let $M\subset \C P^{m+r}$ be a complete intersection, defined by $r$ homogeneous polynomials of degrees $a_1,\ldots,a_r\in\N$.
Then $M$ is a complex submanifold of complex dimension $m$.
Let $h_0\in H^2(\C P^{m+r},\Z)$ be the generator of the cohomology ring of $\C P^{m+r}$ represented by $(2\pi \mathrm{i})^{-1}$ times the K\"ahler form.
Let $i:M \hookrightarrow \C P^{m+r}$ be the inclusion map and put $h:=i^*h_0 \in H^2(M,\Z)$.

We regard $TM$ as a complex vector bundle and denote by $T^\C M$ the complexification of the realification of $TM$, i.e.\ $T^\C M = TM \oplus \overline{TM}$.
The Chern class of $TM$ is given by
\begin{equation}
c(TM) = (1+h)^{m+r+1} (1+a_1h)^{-1}\cdots (1+a_rh)^{-1},
\label{eq:ChernClass}
\end{equation}
see e.g.\ \cite{H95}*{p.~159, eq.~(1)}.
In particular, 
\begin{align*}
c(TM)
&=
(1+(m+r+1)h+\OO(h^2))(1-a_1h+\OO(h^2))\cdots(1-a_rh+\OO(h^2)) \\
&=
1+(m+r+1-(a_1+\ldots+a_r))h + \OO(h^2),
\end{align*}
hence the first Chern class is 
\begin{equation}
c_1(TM) = (m+r+1-(a_1+\ldots+a_r))h .
\label{eq:c1}
\end{equation}
Since the second Stiefel-Whitney class is the mod-$2$ reduction of $c_1$, the manifold $M$ is spin if $m+r+1-(a_1+\ldots+a_r)$ is even.

Since $h>0$ we see that $c_1(TM)\le0$ if and only if $a_1+\ldots+a_r\ge m+r+1$.
In this case, $M$ carries a K\"ahler-Einstein metric with nonpositive scalar curvature by the Calabi-Yau theorem.

Next, we observe
\begin{alignat*}{2}
c(T^\C M)
&=&\,&
c(TM) c(\overline{TM}) \\
&=&&
(1+h)^{m+r+1} (1+a_1h)^{-1}\cdots (1+a_rh)^{-1} \\
&&& \cdot (1-h)^{m+r+1} (1-a_1h)^{-1}\cdots (1-a_rh)^{-1} \\
&=&&
(1-h^2)^{m+r+1} (1-a_1^2h^2)^{-1}\cdots (1-a_r^2h^2)^{-1} 
\end{alignat*}
and thus the Pontryagin class of $M$ is given by
\begin{equation}
p(TM) = (1+h^2)^{m+r+1} (1+a_1^2h^2)^{-1}\cdots (1+a_r^2h^2)^{-1} .
\label{eq:PontryaginClass}
\end{equation}
Let $H$ be the complex line bundle over $M$ whose first Chern class is given by $h$.
By~\eqref{eq:ChernClass}, 
$$
c(TM\oplus H^{a_1}\oplus\cdots\oplus H^{a_r}) = c(\underbrace{H\oplus\cdots\oplus H}_{m+r+1}).
$$
The Chern character of a vector bundle is determined by its Chern class, except for the $0$-degree term which is given by the rank of the bundle.
Thus, modulo $H^0(M,\Z)$ we get
\begin{align*}
\ch(TM\oplus H^{a_1}\oplus\cdots\oplus H^{a_r})
&=
\ch(TM)+\ch(H^{a_1}) + \ldots + \ch(H^{a_r}) \\
&=
\ch(TM)+ e^{a_1h} + \ldots + e^{a_rh}, \\
\ch(\underbrace{H\oplus\cdots\oplus H}_{m+r+1})
&=
(m+r+1)e^h.
\end{align*}
Hence
$$
\ch(TM)
=
-1 +(m+r+1)e^h - e^{a_1h} - \ldots - e^{a_rh}
$$
where the $(-1)$-term fixes the $0$-degree part.
We conclude
\begin{alignat*}{2}
\ch(T^\C M) 
&=&& 
\ch(TM) + \ch(\overline{TM}) \\
&=&&
-1 +(m+r+1)e^h - e^{a_1h} - \ldots - e^{a_rh} \\
&&& -1 +(m+r+1)e^{-h} - e^{-a_1h} - \ldots - e^{-a_rh} \\
&=&&\,
2\big( -1 +(m+r+1) \cosh(h) - \cosh(a_1h)- \ldots - \cosh(a_rh) \big).
\end{alignat*}
For $a\in\N$ let $\HH_a := H^a\oplus H^{-a}$, considered as a real bundle of rank $4$.
Its Pontryagin class is given by $p(\HH_a)=1+a^2h^2$.
By \eqref{eq:PontryaginClass},
$$
p(TM\oplus \HH_{a_1} \oplus\cdots\oplus \HH_{a_r}) = p(\underbrace{\HH_1\oplus\cdots\oplus \HH_1}_{m+r+1}),
$$
hence 
\begin{align*}
\Ahat(TM)
&=
\Ahat(\HH_1)^{m+r+1}\cdot \Ahat(\HH_{a_1})^{-1}\cdots \Ahat(\HH_{a_r})^{-1} \\
&=
\left(\frac{h/2}{\sinh(h/2)}\right)^{m+r+1}\cdot\frac{\sinh(a_1h/2)}{a_1h/2}\cdots\frac{\sinh(a_rh/2)}{a_rh/2}.
\end{align*}
Since $\<h^m,[M]\>=a_1\cdots a_r$ (see e.g.\ \cite{H95}*{p.~160}) we find that
\begin{align}
\<&\Ahat(TM)\ch(T^\C M),[M]\> = \notag \\
&
\coeff\bigg(h^m,\frac{2(\frac{h}{2})^{m+1}}{\sinh(\frac{h}{2})^{m+r+1}} \prod_{j=1}^r \sinh\big(\frac{a_jh}{2}\big)
\Big( (m+r+1) \cosh(h) -1 - \sum_{j=1}^r\cosh(a_jh) \Big)
\bigg) .
\label{CharNumberGeneral}
\end{align}
Here $\coeff(h^m,f(h))$ denotes the coefficient of $h^m$ in the power series $f(h)$. (This notation may seem
extraneous, but will be useful below in \S 5.) 

We analyze this coefficient in the case of hypersurfaces ($r=1$).

\begin{lemma}\label{polynomial1}
Let $m\in\N$ be even.
Then
\begin{align}
\coeff\bigg(h^m,
\frac{2(\frac{h}{2})^{m+1}}{\sinh(\frac{h}{2})^{m+2}} \sinh\Big(\frac{ah}{2}\Big)
\Big((m+2) \cosh(h) -1 - \cosh(ah) \Big)
\bigg)
\label{eq:CharNumberHypersurface}
\end{align} 
is a polynomial in $a$ of degree $m+1$.
\end{lemma}

\begin{proof}
The power series under consideration is of the form
$$
a\cdot(f_1(h)g_1(ah)-f_2(h)g_2(ah))
$$
where $f_1(x),g_1(x),f_2(x),g_2(x)\in\Q\llbracket x\rrbracket$ are even power series.
Indeed, this holds with the choice
\begin{align*}
f_1(x) &= \frac{2(x/2)^{m+2}}{\sinh(x/2)^{m+2}}((m+2) \cosh(x) -1), \\
g_1(x) &= \frac{\sinh(x/2)}{x/2}, \\
f_2(x) &= \frac{2(x/2)^{m+2}}{\sinh(x/2)^{m+2}}, \\
g_2(x) &= \frac{\sinh(x/2)}{x/2}\cosh(x).
\end{align*}
Clearly, the coefficient of $h^m$ in the power series $f_j(h)g_j(ah)$ is a polynomial in $a$ of degree at most $m$.
To see that the degree equals $m$, we determine the coefficient of $a^mh^m$ in $f_j(h)g_j(ah)$ when considered as a power series in the two variables $h$ and $a$.
It is given by $f_j(0)$ times the coefficient of $x^m$ in $g_j(x)$.

For $j=1$ we get 
$$
2(m+1)\cdot \frac{1}{2^m (m+1)!} = \frac{2^{1-m}}{m!}.
$$
For $j=2$ we find, using $g_2(x)=\frac{\sinh(3x/2)-\sinh(x/2)}{x}$,
$$
2\cdot \left(\frac{(3/2)^{m+1}}{(m+1)!}-\frac{(1/2)^{m+1}}{(m+1)!}\right)
=
2^{-m}\frac{3^{m+1}-1}{(m+1)!}.
$$
Thus the coefficient of $a^{m+1}$ in the coefficient of $h^m$ in \eqref{eq:CharNumberHypersurface} is given by
$$
\frac{2^{1-m}}{m!} - 2^{-m}\frac{3^{m+1}-1}{(m+1)!}
=
2^{-m}\frac{2m+3-3^{m+1}}{(m+1)!}
$$
and hence does not vanish.
\end{proof}

\begin{cor}\label{cor:RSlarge}
Let $m\in\N$ be even and $C>0$.
Then there exists a simply connected compact K\"ahler-Einstein spin manifold $M$ with negative scalar curvature and complex dimension $m$ such that $\RS(M)>C$.
\end{cor}

\begin{proof}
Let $M^m_a$ be a smooth complex hypersurface in $\C P^{m+1}$ of degree $a$.
Then $M^m_a$ is compact and simply connected by the Lefschetz hyperplane theorem.
By Lemma~\ref{polynomial1}, $\<\Ahat(TM^m_a)\ch(T^\C M^m_a),[M^m_a]\>$ is a polynomial in $a$ of degree $m+1$.
Thus we may choose $a$ such that
\begin{enumerate}[\myicon]
\item 
$a$ is even, hence $M^m_a$ is spin;
\item
$a>m+2$, hence $c_1(TM^m_a)<0$;
\item
$a$ is so large that $|\<\Ahat(TM^m_a)\ch(T^\C M^m_a),[M^m_a]\>| > C$.
\end{enumerate}
Since $c_1(TM^m_a)<0$ the Calabi-Yau theorem implies that $M^m_a$ carries a K\"ahler-Einstein metric with negative scalar curvature. 
Applying Theorem~\ref{thm:LowerBoundRS} to $M^m_a$ with this metric yields the claim.
\end{proof}

The manifolds we used in the proof of Corollary~\ref{cor:RSlarge} are all hypersurfaces of $\C P^{m+1}$.
One finds more examples using complete intersections of higher codimension.

\begin{lemma}\label{polynomial2}
If $r \geq 1$ is arbitrary and $m$ is even, then the coefficient of $h^m$ in \eqref{CharNumberGeneral} is a nontrivial symmetric polynomial of degree $m+1$ in $\vec a = (a_1, \ldots, a_m)$. 
Its value on $(a_1, 1, 1, \ldots, 1)$ is equal to the coefficient in \eqref{eq:CharNumberHypersurface} evaluated at $a_1$.   
When $m$ is odd, the coefficient of $h^m$ equals $0$.
\end{lemma}
\begin{proof}
Expand all the terms in \eqref{CharNumberGeneral} in Taylor series to obtain a quotient of two formal series in $h$. Inverting
in $\mathbb C[a_1, \ldots, a_r][[h]]$ (the field of formal power series in $h$ with coefficients polynomials in the vector $\vec a$) 
shows that the coefficient of $h^m$ is a polynomial in the entries of $\vec a$. This polynomial is obviously symmetric.
An examination of the degrees of each of the series involved shows that the series contains only even powers of $h$. In other words, 
we only need consider the case where $m$ is even.   It is also not hard to see that the highest power of $\vec a$ in the coefficient of
$h^m$ is $m+1$.  Finally, writing the expression in \eqref{CharNumberGeneral} by $F_r(a_1, \ldots, a_r)$, then we also see that
$F_r(a_1, \ldots, a_s, 1, \ldots, 1) = F_s(a_1, \ldots, a_s)$.  In particular, evaluating $F_r$ on $(a_1, 1, \ldots, 1)$, we reduce
to the case considered in Lemma \ref{polynomial1}, and we have in fact calculated the coefficient of the highest term $a_1^{m+1}$ 
explicitly.  This proves the result. 
\end{proof}

\section{Products}

Using suitable complex submanifolds in complex projective spaces, we found compact Riemannian spin manifolds of any real dimension 
divisible by $4$ with arbitrarily large space of Rarita-Schwinger fields. To treat the remaining dimensions we consider products of these 
hypersurfaces with flat tori. 

We start with some general considerations about spinor fields on product manifolds, see \cite[Sec.~1]{B98} for details.
Let $X$ and $Y$ be Riemannian spin manifolds of dimension $n$ and $m$, respectively. We will only need the case 
when $n$ is even which we assume from now on.  If $M := X\times Y$, then $TM$ is naturally identified with 
$\pi_1^*TX\oplus \pi_2^*TY$ and the spinor bundle $\Sigma M$ with $\pi_1^*\Sigma X \otimes \pi_2^*\Sigma Y$.
Here $\pi_1:M\to X$ and $\pi_2:M\to Y$ are the obvious projections. If $m$ is also even, then
\[
\Sigma^\pm M = (\pi_1^*\Sigma^+ X \otimes \pi_2^*\Sigma^\pm Y)\oplus  (\pi_1^*\Sigma^- X \otimes \pi_2^*\Sigma^\mp Y)
\]
Clifford multiplication is given by 
\begin{equation}
\gamma_M((v \oplus w)\otimes \varphi \otimes \psi) = \gamma_X(v\otimes \varphi)\otimes \psi + \eps \varphi\otimes \gamma_Y(w\otimes \psi)
\label{eq:CliffordProduct}
\end{equation}
where $\eps=1$ if $\varphi\in \pi_1^*\Sigma^+ X$ and $\eps=-1$ if $\varphi\in \pi_1^*\Sigma^- X$.
If $\Phi = v \otimes \varphi \in \widehat\Sigma_xX$ and $\psi\in\Sigma_yY$, then $\Phi \otimes \psi \in \pi_1^* T_xX \otimes
\pi_1^* \Sigma_xX \otimes \pi_2^* \Sigma_yY$. By \eqref{eq:CliffordProduct}, $\gamma_M(\Phi\otimes\psi) = \gamma_X(\Phi)\otimes \psi = 0$, 
and hence $\Phi\otimes\psi\in \widehat\Sigma_{(x,y)}M$. This proves that
$$
\pi_1^*\widehat\Sigma X \otimes \pi_2^*\Sigma Y \subset \widehat\Sigma M.
$$

Now let $\Phi$ be a Rarita-Schwinger field on $X$ and $\psi$ a parallel spinor on $Y$. Using a local orthonormal frame 
$e_1,\ldots,e_n,e_{n+1},\ldots,e_{n+m}$ of $M$ where the $e_i$ are tangent to $X$ ($i\le n$) and the $e_{n+i}$ to $Y$,  we compute 
\begin{align*}
\DD_M(\pi_1^*\Phi\otimes\pi_2^*\psi)
&=
\sum_{i=1}^{n}\gamma_M\big(e_i\otimes (\pi_1^*\nabla_{e_i}^X\Phi\otimes\pi_2^*\psi)\big) + 
\sum_{i=n+1}^{n+m}\gamma_M\big(e_i\otimes (\pi_1^*\Phi\otimes\pi_2^*\nabla_{e_i}^Y\psi)\big) \\
&=
\pi_1^*\DD_X(\Phi)\otimes\pi_2^*\psi  =0.
\end{align*}
Thus $\pi_1^*\Phi\otimes\pi_2^*\psi$ is a Rarita-Schwinger field on $M$.  Writing
$$
\PS(Y) := \dim\{\psi\in\Gamma(\Sigma Y) \mid \nabla\psi=0\},
$$
and when $m$ is even, 
$$
\PS^\pm(Y) := \dim\{\psi\in\Gamma(\Sigma^\pm Y) \mid \nabla\psi=0\},
$$
then we have shown:

\begin{prop}\label{prop:parallel}
Let $X$ and $Y$ be Riemannian spin manifolds, with $X$ even-dimensional.  
Then
$$
\RS(X\times Y) \ge \RS(X)\cdot\PS(Y).
$$
Moreover, if $Y$ is also even-dimensional, then
\usetagform{simple}
\begin{equation}
\RS^\pm (X\times Y) \ge \RS^+(X)\cdot\PS^\pm(Y) + \RS^-(X)\cdot\PS^\mp(Y). 
\tag{$\Box$}
\end{equation}
\usetagform{default}
\end{prop}

As an immediate consequence, we can now produce manifolds of any dimension $n \geq 4$ with an arbitrarily
large dimensional space of Rarita-Schwinger fields:
\begin{cor}
For any $n \geq 4$ and $C>0$,  there exists a compact Riemannian spin manifold $M$ of dimension $n$ 
such that $\RS(M)>C$.
\end{cor}
\begin{proof}
Write $n=2m +k$ where $m$ is even and $k\in\{0,1,2,3\}$.
Let $X$ be a Kähler-Einstein manifold of complex dimension $m$ as in Corollary~\ref{cor:RSlarge} such that $\RS(X)>C$.
If $k=0$ then we simply take $M:=X$, while in the other cases, we take $M:=X\times T^k$ where $T^k$ is a flat torus.
Endow $T^k$ with the (unique) spin structure for which it has nontrivial parallel spinors. Then by Proposition~\ref{prop:parallel}, 
$\RS(M)\ge\RS(X)>C$.
\end{proof}

\section{Calabi-Yau manifolds}

The simplest examples of Ricci flat manifolds with Rarita-Schwinger fields are provided by flat tori, again equipped with the spin structure which admits parallel spinors.
For flat metrics, a section of $\Sh M$ is a Rarita-Schwinger field if and only if it is parallel. 
Thus for $M=T^n$ a flat torus, $\RS(T^n)$ equals the rank of $\Sh T^n$, i.e.\ $\RS(T^n)=(n-1)\cdot 2^{[\nicefrac{n}{2}]}$.

Complete intersections yield more interesting examples. 
For the sake of simplicity, we focus on hypersurfaces. 
Thus, using the previous notation, set $r=1$ and $a=m+2$; a hypersurface $M^m\subset \C P^{m+1}$ of degree $a$ has vanishing first Chern class by \eqref{eq:c1}. 
Thus $M^m$ is spin and by the Calabi-Yau theorem carries a Ricci flat K\"ahler metric. 

According to \eqref{CharNumberGeneral} we have
\begin{align*}
\<&\Ahat(TM^m)\ch(T^\C M^m),[M^m]\> = \\
&\coeff\bigg(h^m,\frac{2(\frac{h}{2})^{m+1}}{\sinh(\frac{h}{2})^{m+2}} \sinh\Big(\frac{(m+2)h}{2}\Big)
\Big((m+2) \cosh(h) -1 - \cosh((m+2)h) \Big)\bigg).
\end{align*}
Since the power series is even, we restrict ourselves to the case of even $m$, i.e., the real dimension of $M^m$ is divisible by $4$.

\begin{lemma}
If $m\in\N$ is even, then the coefficient on the right in the preceding equation equals
\[
-2\bigg[\begin{pmatrix} 2m+3 \\ m+1 \end{pmatrix} +1 - (m+2)^2\bigg] .
\]
\end{lemma}

\begin{proof}
We compute the coefficient in this power series using two applications of the residue theorem. 
Let $\Gamma$ be a loop encircling the origin once counterclockwise in the complex plane.
Using the substitution $z=\exp(h)-1$ we find
\begin{align*}
&\coeff\bigg(h^m,\frac{2(\frac{h}{2})^{m+1}}{\sinh(\frac{h}{2})^{m+2}} \sinh\Big(\frac{(m+2)h}{2}\Big)
\Big((m+2) \cosh(h) -1 - \cosh((m+2)h) \Big)\bigg) \\
&=
2^{-m}\frac{1}{2\pi i}\int_\Gamma \frac{\sinh\big(\frac{(m+2)h}{2}\big)}{\sinh(\frac{h}{2})^{m+2}} 
\Big((m+2) \cosh(h) -1 - \cosh((m+2)h) \Big) dh \\
&=
2^{-m}\frac{1}{2\pi i}\int_\Gamma \frac{\frac12 [(1+z)^{\frac{m+2}{2}}-(1+z)^{-\frac{m+2}{2}}]}{(\frac12 [(1+z)^{\frac{1}{2}}-(1+z)^{-\frac{1}{2}}])^{m+2}} \times \\
&\quad\quad
\times\Big(\tfrac{m+2}{2}(1+z + (1+z)^{-1}) -1 - \tfrac12 ((1+z)^{m+2} + (1+z)^{-(m+2)}) \Big) \frac{dz}{1+z}\\
&=
\frac{1}{2\pi i}\int_\Gamma \frac{(1+z)^{m+2}-1}{[(1+z)-1]^{m+2}} \times \\
&\quad\quad
\times\Big((m+2)(1+z + (1+z)^{-1}) -2 - ((1+z)^{m+2} + (1+z)^{-(m+2)}) \Big) \frac{dz}{1+z} \\
&=
\coeff\bigg(z^{m+1}, ((1+z)^{m+2}-1)\times \\
&\quad
\times\Big((m+2)(1+z + (1+z)^{-1}) -2 - ((1+z)^{m+2} + (1+z)^{-(m+2)}) \Big)(1+z)^{-1}\bigg)\\
&=
\coeff\bigg(z^{m+1},-(1+z)^{2m+3} + (m+2)(1+z)^{m+2} -(1+z)^{m+1} + \\
&\quad\quad
+(m+2)(1+z)^m -(m+2) + (1+z)^{-1} - (m+2)(1+z)^{-2} +(1+z)^{-(m+3)}  \bigg) \\
&=
-\begin{pmatrix}2m+3 \\ m+1 \end{pmatrix} + (m+2) \begin{pmatrix}m+2 \\ m+1 \end{pmatrix} -1 + 0 - 0 +  \begin{pmatrix}-1 \\ m+1 \end{pmatrix} -(m+2)\begin{pmatrix}-2 \\ m+1 \end{pmatrix} + \\
&\quad\quad\quad
+ \begin{pmatrix}-(m+3) \\ m+1 \end{pmatrix} \\
&=
-2\bigg[\begin{pmatrix} 2m+3 \\ m+1 \end{pmatrix} +1 - (m+2)^2\bigg] .
\qedhere
\end{align*}
\end{proof}

This lemma and Theorem~\ref{thm:LowerBoundRS} combine to give

\begin{cor}\label{cor:CYlowerbound}
Let $m\in\N$ be even.
Then there exists a simply connected compact Calabi-Yau manifold of complex dimension $m$ such that 
\usetagform{simple}
\begin{equation}
\RS(M^m)
\ge
2\bigg[\begin{pmatrix} 2m+3 \\ m+1 \end{pmatrix} +1 - (m+2)^2\bigg] - 2^{m/2}.
\tag{$\Box$}
\end{equation}
\usetagform{default}
\end{cor}

\begin{remark}
For $m=2$, $M^m$ is the K3-surface.
In this case Corollary~\ref{cor:CYlowerbound} yields the lower bound $\RS(M^2)\ge 38$.
This is sharp; indeed $\RS(\mathrm{K3})=38$, see Example~(1) to Proposition~4.6 in \cite{HS2018}.
We do not know whether the bound is also sharp in higher dimensions.
\end{remark}

\begin{remark}
It is easy to see that 
$$
2\bigg[\begin{pmatrix} 2m+3 \\ m+1 \end{pmatrix} +1 - (m+2)^2\bigg] - 2^{m/2}
>
(2m-1)\cdot 2^{m}.
$$
Thus our simply connected compact Calabi-Yau manifolds have more linearly independent Rarita-Schwinger fields than flat tori of the same dimension.
This is illustrated by the following table for low dimensions:

\begin{center}
\begin{tabular}{|c|r|r|}
\hline
$m$ & \multicolumn{1}{c|}{$\RS(M^m)\ge$} & \multicolumn{1}{c|}{$\RS(T^{2m})=$} \\ \hline\hline
2 & 38 & 12 \\ \hline
4 & 850 & 112 \\ \hline
6 & 12,736 & 704 \\ \hline
8 & 184,542 & 3,840 \\ \hline
10 & 2,703,838 & 19,456 \\ \hline
12 & 40,116,146 & 94,208 \\ \hline
14 & 601,079,752 & 442,368 \\ \hline
16 & 9,075,134,398 & 2,031,616 \\ \hline
18 & 137,846,527,510 & 9,175,040 \\ \hline
20 & 2,104,098,961,730 & 40,894,464 \\ \hline
22 & 32,247,603,679,902 & 180,355,072 \\ \hline
24 & 495,918,532,942,658 & 788,529,152 \\ \hline
26 & 7,648,690,600,750,682 & 3,422,552,064 \\ \hline
28 & 118,264,581,564,843,242 & 14,763,950,080 \\ \hline
30 & 1,832,624,140,942,555,720 & 63,350,767,616 \\ \hline
\end{tabular}
\medskip
\captionof{table}{Linearly independent Rarita-Schwinger fields on Calabi-Yau manifolds}
\end{center}
\end{remark}


\end{document}